\documentclass[a4paper, 11pt]{amsart}

\usepackage{amssymb,amsmath}
\usepackage[english]{babel}
\usepackage[latin1]{inputenc}
\usepackage{comment}

\newtheorem{theorem}{Theorem}[section]
\newtheorem{lemma}[theorem]{Lemma}

\theoremstyle{definition}
\newtheorem{definition}[theorem]{Definition}

\numberwithin{equation}{section}

\title[A pointwise characterization of BV functions on metric spaces]{A pointwise characterization of functions of bounded variation on metric spaces}

\author{Panu Lahti}
\address{Department of Mathematics,
P.O. Box 11100,
FI-00076 Aalto University, Finland}
\email{panu.lahti@aalto.fi}
\author{Heli Tuominen}
\address{Department of Mathematics and Statistics,
P.O. Box 35 (MaD),
FI-40014 University of Jyv\"askyl\"a, Finland}
\email{heli.m.tuominen@jyu.fi}
\date{\today}

\newcommand\rn{\mathbb R^n}

\newcommand\z{\mathbb Z}
\newcommand\grad{\nabla}

\newcommand\ph{\varphi}
\newcommand\M{\operatorname{\mathcal M}}
\newcommand\diam{\operatorname{diam}}
\newcommand\dist{\operatorname{d}}

\providecommand{\sob}[1]{W^{1,p}({#1})}
\providecommand{\BV}[1]{BV({#1})}

\providecommand{\lloc}[1]{L^1_{\text{loc}}(#1)}

\providecommand{\ch}[1]{\text{\raise 2pt \hbox{$\chi$}\kern-0.2pt}_{#1}}

\providecommand{\vint}[1]{\mathchoice
          {\mathop{\vrule width 5pt height 3 pt depth -2.5pt
                  \kern -9pt \kern 1pt\intop}\nolimits_{\kern -5pt{#1}}}%
          {\mathop{\vrule width 5pt height 3 pt depth -2.6pt
                  \kern -6pt \intop}\nolimits_{\kern -3pt{#1}}}%
          {\mathop{\vrule width 5pt height 3 pt depth -2.6pt
                  \kern -6pt \intop}\nolimits_{\kern -3pt{#1}}}%
          {\mathop{\vrule width 5pt height 3 pt depth -2.6pt
                  \kern -6pt \intop}\nolimits_{\kern -3pt{#1}}}}

\begin{document}

\thanks{Part of this research was conducted during the visit of the second
author to Forschungsinstitut f\"ur Mathematik of ETH Z\"urich,  and she wishes to 
thank the institute for the kind hospitality.
The second author was supported by
the Academy of Finland, grant no.\ 135561.}

\begin{abstract}
We give a new characterization of the space of functions of bounded variation in terms of a pointwise inequality connected to the maximal function of a measure. The characterization is new even in Euclidean spaces and it holds also in general metric spaces. 
\end{abstract}

\maketitle

\section{Introduction}
There are several equivalent definitions for functions of bounded variation in Euclidean spaces. Two of the most common, which we recall below, cannot be generalized to metric spaces because they make use of smooth functions and (weak) derivatives. An integrable function $u$ is a function of bounded variation, $u\in BV(\rn)$, if 
\[
\|Du\|(\rn)=
\sup\Big\{\int_{\rn}u\operatorname{div} \ph\,dx:
\ph\in C^{1}_{c}(\rn,\rn), \|\ph\|_{\infty}\le 1 \Big\}<\infty,
\]
or, equivalently, if there exist real finite measures $\mu_{1},\dots,\mu_{n}$ such that 
\[
\int_{\rn}uD_{i}\ph\,dx=-\int_{\rn}\ph\,d\mu_{i}
\text{ for all }\ph\in  C^{1}_{c}(\rn),\; i=1,\dots,n,
\]
that is, the weak gradient $Du=\mu$ of $u$ is an $\rn$-valued measure with finite total variation $|Du|(\rn)$. The above definitions are equivalent to the following one, based on a relaxation procedure using Lipschitz functions. A function $u\in L^{1}(\rn)$ belongs to $BV(\rn)$, if 
\begin{equation}\label{BV relaxation rn}
L(u)=
\inf\Big\{\liminf_{i\to\infty}\int_{\rn}|\grad u_{i}|\,dx:
u_{i}\in \operatorname{Lip}(\rn),\, u_{i}\to u\text{ in }L^{1}(\rn) \Big\}<\infty.
\end{equation}
The definition given by \eqref{BV relaxation rn} has been generalized to a metric measure space by using the local Lipschitz constant \eqref{lip} in the place of the gradient by Ambrosio in \cite{A1} and Miranda in \cite{Mi}. This definition together with the doubling property of the measure and the validity of a $(1,1)$-Poincar\'e inequality provides a rich theory of functions of bounded variation in metric spaces, see for example 
\cite{A1}, \cite{A}, \cite{AMP}, \cite{Cam}, \cite{HakKi}, \cite{KKST1}, \cite{KKST3}, \cite{KKST2}, \cite{Mi}.
Properties of  $BV$ functions in $\rn$ can be studied from the monographs \cite{AFP} (contains a historical overview in Section 3.12), \cite{EG}, \cite{F}, \cite{Gi}, \cite{Z}.  

In this paper, motivated by the characterization of the Sobolev space $W^{1,1}(\rn)$ 
given by Haj\l asz in \cite{Hj3},  we give a new characterization of functions of bounded variation in metric spaces using a pointwise estimate. 
The characterization is  new even in the classical setting. For the notation and definitions used in the introduction and throughout the paper, see Section \ref{preliminaries}.

Before giving the characterization, we recall the inequality behind the
Sobolev spaces $M^{1,p}(X)$, where $X=(X,\dist,\mu)$ is a metric measure space.
For $1<p<\infty$, the function $u\in L^p(\rn)$ belongs to $\sob{\rn}$ if
and only if there is a function $0\le g\in L^p(\rn)$ such that the
pointwise inequality 
\begin{equation}\label{m1p rn}
|u(x)-u(y)|\le|x-y|\bigl(g(x)+g(y)\bigr)
\end{equation}
holds for almost all $x,y\in\rn$, see \cite{Hj1}. The
validity of \eqref{m1p rn} for $u\in\sob{\rn}$ follows from the inequality
\begin{equation}\label{u-u M}
|u(x)-u(y)|\le C(n)|x-y|\bigl[\M_{2|x-y|}|\grad u|(x)+
\M_{2|x-y|}|\grad u|(y)\bigr]
\end{equation}
for almost all $x,y\in\rn$, which holds for all $1\le p<\infty$, and the $L^p$-boundedness of the Hardy-Littlewood maximal operator $\M$ for $p>1$, see for example \cite{Hj1}. 
The boundedness is essential; for a function $u\in W^{1,1}(\rn)$ there is not
necessarily any integrable function $g$ such that inequality
\eqref{m1p rn} holds, see \cite{Hj3}. In \cite{Hj3}, Haj\l asz gave
the following characterization of $W^{1,1}(\rn)$ using a pointwise
estimate with maximal functions on its right-hand side. 

\begin{theorem}[{\cite[Theorem 4]{Hj3}}]\label{w11 pointwise}
Let $u\in L^1(\rn)$. Then $u\in W^{1,1}(\rn)$ if and only if there
exists a function $0\le g\in L^1(\rn)$ and a constant
$\sigma\ge1$ such that the pointwise inequality 
\begin{equation}\label{w11}
|u(x)-u(y)|\le|x-y|\bigl[\M_{\sigma|x-y|}g(x)+\M_{\sigma|x-y|}g(y)\bigr]
\end{equation}
holds for almost all $x,y\in\rn$.    
\end{theorem}

%In the proof of the ``if'' part of Theorem \ref{w11 pointwise}, one first shows that inequality \eqref{w11} implies a Poincar\'e inequality
%\[
%\vint B |u-u_{B}|\,dx
%\le C r\vint{3\sigma B}g\,dx
%\] 
%for all balls $B=B(x,r)$. The validity of this Poincar\'e inequality implies that $u$ belongs to the Sobolev space $W^{1,1}(\rn)$, and that $|\grad u|\le Cg$ almost everywhere.

In the metric setting, we can characterize Newtonian functions by a similar pointwise inequality, provided the space supports the $(1,p)$-Poincar\'e inequality \eqref{eq: poincare}.  Recall that Newtonian spaces are a generalization of Sobolev spaces to metric spaces using upper gradients, see \cite{S}. For any $p>0$, a $(1,p)$-Poincar\'e inequality for a pair 
$u\in\lloc{X}$ and a measurable function $g\ge0$ implies, using a standard chaining argument, a pointwise inequality of the same
type as \eqref{w11},
\begin{equation}\label{pointwise p-poincare}
|u(x)-u(y)|\le C\dist(x,y)  
 \Bigl[\bigl( \M_{2\tau\dist(x,y)}g^p(x)\bigr)^{1/p}
 +\bigl( \M_{2\tau\dist(x,y)}g^p(y)\bigr)^{1/p}\Bigr]
\end{equation}
for $\mu$-almost all $x,y\in X$, see \cite[Theorem 3.2]{HjK}.
A converse holds when $p>s/(s+1)$, where $s$ is the doubling dimension: if the pair $u,g\in L^p(X)$ satisfies the  pointwise inequality \eqref{pointwise p-poincare}, it also satisfies a $(1,p)$-Poincar\'e inequality, see \cite[Theorem 9.5]{Hj2}. Furthermore, if $p \ge 1$ and $X$ is complete, then according to \cite[Theorem 11.2]{Hj2}, it follows that $u$ belongs to the Newtonian space $N^{1,p}(X)$. Thus the pointwise inequality \eqref{pointwise p-poincare} for $u,g\in L^p(X)$ characterizes the space $N^{1,p}(X)$ for any $1 \le p < \infty$.

%\cite[Theorem 4.5]{Ko} and \cite[Theorem 4.9]{S}, it follows that $u$ belongs to the Newtonian space $N^{1,p}(X)$, provided that $X$ supports a $(1,q)$-Poincar\'e inequality for some $q<p$. Newtonian spaces are a generalization of Sobolev spaces to metric spaces using upper gradients, see \cite{S}. In the case that $X$ supports a $(1,p)$-Poincar\'e inequality, it in fact also supports the required $(1,q)$-Poincar\'e inequality, according to \cite[Theorem 1.0.1]{KZ} and \cite[Theorem 2]{K}. In this case, \eqref{pointwise p-poincare} thus characterizes Newtonian spaces.

For $BV$ functions, Poincar\'e inequality \eqref{eq: poincareBV} and the same proof as in 
\cite[Theorem 3.2]{HjK} give a similar estimate as \eqref{pointwise p-poincare} for the oscillation 
of a function. Namely, if $u\in\BV{X}$, then for $\mu$-almost all $x,y\in X$,
\begin{equation}\label{pointwise from poincare}
|u(x)-u(y)|\le C\dist(x,y)  
\big[\M_{2\tau\dist(x,y),\|Du\|}(x)+\M_{2\tau\dist(x,y),\|Du\|}(y)\big],
\end{equation}
where the constant $C>0$ depends only on the doubling constant $c_d$ and on the constants of the Poincar\'e inequality. Here $\M_{2\tau d,\|Du\|}$ is the restricted maximal function 
\eqref{maximal function measure} of the measure $\|Du\|$. 
In Theorem \ref{poincare from pointwise}, we show that a similar pointwise inequality \eqref{pointwise} with a maximal function of a measure implies a Poincar\'e type inequality \eqref{nu poincare} with the same measure on the right hand side. This, together with Theorem 
\ref{char miranda} by Miranda,  shows that $u$ is a function of bounded variation.

\begin{theorem}[{\cite[Theorem 3.8]{Mi}}]\label{char miranda}
Let $X$ be a complete, doubling metric measure space that supports a $(1,1)$-Poincar\'e inequality. 
Let $u\in L^1(X)$. 
Then $u\in BV(X)$ if and only if there exist constants $C_{1}>0$ and $\eta >0$ and a positive, finite measure 
$\nu$ such that
\[
\int_{B}|u-u_{B}|\,d\mu\le C_{1}r\nu(\eta B)
\]
for each ball $B(x,r)$. Moreover, $\|Du\|\le C\nu$, with $C=C(C_{1}, c_{d},\eta)$.
\end{theorem}

We obtain our characterization by combining \eqref{pointwise from poincare}, Theorem 
\ref{poincare from pointwise}, and Theorem \ref{char miranda}. Although the characterization is new even in Euclidean spaces, we formulate it only in the general metric space setting. 
\begin{theorem}\label{char pointwise}
Let $X$ be a complete, %$Q$-regular, geodesic 
doubling metric measure space that supports 
a $(1,1)$-Poincar\'e inequality. Let $u\in L^1(X)$.  
Then $u\in BV(X)$ if and only if there exists a positive, finite measure $\nu$ and 
constants $\sigma\ge1$ and $C_{0}>0$ such that the inequality 
\[
|u(x)-u(y)|\le C_{0}\dist(x,y)  
\big[\M_{\sigma\dist(x,y),\nu}(x) + \M_{\sigma\dist(x,y),\nu}(y)\big]
\]
holds for $\mu$-almost all $x,y\in X$. Moreover, $\|Du\|\le C\nu$, where $C$ only depends on $C_{0}$, $\sigma$, the doubling constant of the measure, and the constants in the $(1,1)$-Poincar\'e inequality.
\end{theorem}

\section{\label{preliminaries}Notation and preliminaries}

We assume that $X=(X,\dist,\mu)$ is a metric measure space
equipped with a metric $\dist$ and a Borel regular, {\it doubling} outer
measure $\mu$. The doubling property means that there is a fixed constant
$c_d>0$, called the {\it doubling constant} of $\mu$, such that 
\begin{equation}\label{doubling measure}
\mu(2B)\le c_d\mu(B)
\end{equation}
for every ball $B=B(x,r)=\{y\in X:\dist(y,x)<r\}$. Here $tB=B(x,tr)$. 
We assume that the measure of every open set is positive and
that the measure of each bounded set is finite. 
The doubling condition gives an upper bound for the dimension of $X$.
By this we mean that there is a constant $C=C(c_d)>0$ and an exponent $s\geq 0$ such that 
\begin{equation}\label{eq: d dim}
\frac{\mu(B(y,r))}{\mu(B(x,R))}\ge C\Bigl(\frac rR\Bigr)^s
\end{equation}
whenever $0<r\le R<\diam (X)$, $x\in X$, and $y\in B(x,R)$. Inequality
\eqref{eq: d dim} holds certainly with $s=\log_2c_d$ (but it may hold for some smaller exponents as well). We call $s$ the doubling dimension of $X$.

We also assume that $X$ is complete; recall that a metric space with a doubling measure is complete 
if and only if the space is proper, that is, closed and bounded sets are compact. When we say that an inequality such as \eqref{pointwise p-poincare} holds for $\mu$-almost all $x,y\in X$, we mean that there is a set $E \subset X$ such that the property holds for all $x,y\in X \setminus E$, and $\mu(E)=0$. 

The restricted Hardy-Littlewood maximal function of a locally integrable function
$u$ is 
\begin{equation}\label{maximal function}
\M_Ru(x)=\sup_{0<r\le R}\,\vint{B(x,r)}|u(y)|\,d\mu(y),
\end{equation}
where $u_B=\vint{B}u\,d\mu=\mu(B)^{-1}\int_Bu\,d\mu$ is the integral average of $u$ over $B$.     
For $R=\infty$, $\M_{\infty}u$ is the usual Hardy-Littlewood maximal
function $\M u$.

Similarly, the restricted maximal function of a positive, finite measure $\nu$ is 
\begin{equation}\label{maximal function measure}
\M_{R,\nu}(x)=\sup_{0<r\le R}\,\frac{\nu(B(x,r))}{\mu(B(x,r))}.  
\end{equation}
For $R=\infty$, we write $\M_{\nu}$.

A curve is a rectifiable continuous mapping from a compact interval to $X$. 
A nonnegative Borel function $g$ on $X$ is an upper gradient 
of an extended real valued function $u$
on $X$, if for all curves $\gamma$ in $X$, we have
\begin{equation} \label{ug-cond}
|u(x)-u(y)|\le \int_\gamma g\,ds,
\end{equation}
whenever both $u(x)$ and $u(y)$ are finite, and 
$\int_\gamma g\, ds=\infty $ otherwise.
Here $x$ and $y$ are the end points of $\gamma$.
If $g$ is a nonnegative measurable function on $X$
and \eqref{ug-cond} holds for almost every curve with respect to the 1-modulus,
then $g$ is a $1$-weak upper gradient of~$u$. For the concept of modulus in metric spaces, see \cite{HeKo}.
A natural upper gradient for a Lipschitz function $u$ is the local Lipschitz constant 
\begin{equation}\label{lip}
\operatorname{Lip}u(x)=\liminf_{r\to 0}\sup_{y\in B(x,r)}\frac{|u(x)-u(y)|}{\dist(x,y)}.
\end{equation}
Next we recall the definition of functions
of bounded variation on metric spaces, given by Miranda in \cite{Mi}. 

\begin{definition}
For $u\in L^1_{\text{loc}}(X)$, we define
\[
\|Du\|(X)
=\inf\Big\{
\liminf_{i\to\infty}\int_X\operatorname{Lip}u_i\,d\mu: u_i\in \text{Lip}_{\text{loc}}(X),
u_i\to u\text{ in } L^1_{\text{loc}}(X)\Big\},
\]
and we say that a function $u\in L^1(X)$ is of bounded variation, 
$u\in BV(X)$, if $\|Du\|(X)<\infty$. Note that replacing the Lipschitz constants with $1$-weak upper gradients in the definition yields the same space.
\end{definition}
 
We say that $X$ supports a (weak) $(1,p)$-Poincar\'e inequality, $0< p<\infty$,
if there exist constants $c_P>0$ and $\tau \ge 1$ such that for all
balls $B=B(x,r)$, all locally integrable functions $u$,
and all $p$-weak upper gradients $g$ of $u$, we have 
\begin{equation}\label{eq: poincare}
\vint B |u-u_B|\, d\mu 
\le c_P r\Big( \,\vint{\tau B}g^p\,d\mu\Big)^{1/p}.
\end{equation}
If the space supports a $(1,1)$-Poincar\'e inequality, then for
every $u\in BV(X)$ we have
\begin{equation}\label{eq: poincareBV}
\vint B |u-u_B|\, d\mu 
\le c_P r\frac{\|Du\|(\tau B)}{\mu(\tau B)},
\end{equation}
where the constant $c_P$ and the dilation factor $\tau$ are the same as in \eqref{eq: poincare}. Inequality \eqref{eq: poincareBV} follows easily by using \eqref{eq: poincare} for approximating Lipschitz functions in the definition of $BV(X)$.

The characteristic function of a set $E\subset X$ is $\ch{E}$. Both the Euclidean distance and the Lebesgue measure in $\rn$ are denoted by $|\cdot|$. 
In general, $C$ will denote a positive constant whose value is not
necessarily the same at each occurrence.

\section{\label{section pointwise}Pointwise estimate and Poincar\'e
  inequality}
We begin with a geometric lemma.  
Recall that $X$ is {\it a geodesic space} if every two points $x, y\in X$ can be  joined by a curve whose length is equal to $\dist(x,y)$. 

\begin{lemma}\label{small ball inside}
Let $X$ be a geodesic metric space. 
If $B(x_0,R)$ is a ball, $x\in B(x_{0},R)$, and $0<r\le2R$, 
then there is a ball of radius $r/2$ in $B(x,r)\cap B(x_{0},R)$.  
\end{lemma}
\begin{proof}
If $\dist(x,x_0)\ge r/2$, then the assumption that $X$ is geodesic
implies that there is a point $z$ such that $\dist(z,x)=r/2$
and $\dist(z,x_0)=\dist(x,x_0)-r/2$, and hence 
$B(z,r/2)\subset B(x,r)\cap B(x_{0},R)$.

On the other hand, if $\dist(x,x_0)<r/2$, then 
$B(x_0,r/2)\subset B(x,r)\cap B(x_{0},R)$.
\end{proof}

The idea of the proof of the next theorem is from \cite{Hj2} and \cite{Hj3}. 
\begin{theorem}\label{poincare from pointwise}
Let $X$ be a complete, doubling metric measure space that supports 
a $(1,1)$-Poincar\'e inequality. 
Let $u\in\lloc{X}$,  and  let $\nu$ be a positive, finite measure. 
If there are constants $\sigma\ge1$ and $C_{0}>0$ such that the inequality 
\begin{equation}\label{pointwise} 
|u(x)-u(y)|\le C_{0}\dist(x,y)  
\big[\M_{\sigma\dist(x,y),\nu}(x) + \M_{\sigma\dist(x,y),\nu}(y)\big]
\end{equation}
holds for $\mu$-almost all $x,y\in X$, then 
\begin{equation}\label{nu poincare}
 \int_{B}|u-u_{B}|\,d\mu\le Cr\nu(\eta B)
\end{equation}
for each ball $B=B(x,r)$. The constants $C$ and $\eta$ depend only on $C_{0}$, $\sigma$, 
the doubling constant of the measure, and the constants in the $(1,1)$-Poincar\'e inequality.
\end{theorem}

\begin{proof}
Since $X$ supports a $(1,1)$-Poincar\'e inequality, $X$ is
quasiconvex. This means that there is a constant $C \ge 1$, depending only on the doubling constant $c_d$ and the constants in the Poincar\'e inequality, such that every two points
$x,y\in X$ can be connected by a curve $\gamma$ satisfying
$\ell(\gamma) \le C \dist (x,y)$, where $\ell(\gamma)$ is the length of $\gamma$, 
see \cite[Proposition~4.4]{HjK}. 
This implies that $X$ endowed with the length metric 
$\rho(x,y)= \inf \ell(\gamma)$,
where the infimum is taken over all curves connecting $x$
and $y$, is bi-Lipschitz homeomorphic to $X$.
Since the doubling property, the $(1,1)$-Poincar\'e inequality and inequalities  \eqref{pointwise} and \eqref{nu poincare} are invariant under bi-Lipschitz homeomorphisms (cf.\ \cite[Chapter~9]{He}), we may replace the metric $\dist$ by $\rho$ and work with the new space $(X,\rho,\mu)$.
Therefore, throughout the proof we will assume that $\dist$ is
the length metric. Since $X$ is proper, such a metric has the property that every two points
$x,y\in X$ can be connected by a geodesic, that is, a curve whose
length equals $\dist(x,y)$, see \cite[Theorem~3.9]{Hj2}.

Let $B=B(x_0,R)$ be a ball.
We begin the proof by checking what we can assume from $u$ and
$\nu$. Since neither inequality  \eqref{pointwise} nor inequality \eqref{nu poincare} 
change if a constant is added to $u$, we may assume that $\operatorname{ess} \inf_E|u|=0$
for a set $E\subset B$ with $\mu(E)>0$. We will choose the set $E$ later. 

We define $\tau=3\sigma$, and $\lambda=\nu\vert_{\tau B}$. The pointwise estimate \eqref{pointwise} implies that,
%after a modification of $u$ in a set of measure zero if necessary,
\begin{equation}\label{pointwise in B}
|u(x)-u(y)|\le C_0\dist(x,y)  
\big[\M_{\lambda}(x)+\M_{\lambda}(y)\big]
\end{equation}
for almost all $x,y\in B$. We may assume that \eqref{pointwise in B} holds for all $x,y\in B$ because inequality \eqref{nu poincare} with $B$ replaced by $B\setminus F$, where $\mu(F)=0$, implies \eqref{nu poincare} with $B$.

Moreover, we may assume that $\lambda(\tau B)>0$, since otherwise $u$ is constant in $B$, and inequality \eqref{nu poincare} follows.
%By the doubling property of $\mu$, we have for each $x\in B$,
%\begin{equation}\label{iso lambda}
%\M_{\lambda}(x)
%\ge\frac{\lambda(B(x,4\sigma R))}{\mu(B(x,4\sigma R))}
%\ge C\frac{\lambda(\tau B)}{\mu(\tau B)}>0.  
%\end{equation}
For each $k\in\z$, we define
\[
E_k=\big\{x\in B:\M_\lambda(x)\le2^k\big\} 
\quad\text{and}\quad a_k=\sup_{E_k}|u(x)|.
\]
Then $E_{k-1}\subset E_k$ and $a_{k-1}\le a_k$ for each $k$, and 
\begin{equation}\label{u and ak}
\int_B|u-u_B|\,d\mu\le2\int_B|u|\,d\mu
\le2\sum_{k=-\infty}^{\infty}a_k\mu(E_k\setminus E_{k-1}).  
\end{equation}
We will obtain an upper bound for the right hand side of this inequality 
by estimating the values of $a_k$. By the pointwise estimate \eqref{pointwise in B}, the function
$u$ is $C_02^{k+1}$-Lipschitz in $E_k$. Hence, for each $x\in E_k$ and $y\in
E_{k-1}$, we have
\begin{equation}\label{u and ak point}
|u(x)|\le|u(x)-u(y)|+|u(y)|\le C_02^{k+1}\dist(x,y)+a_{k-1}.  
\end{equation}
Our next goal is to find for each $x\in E_k$ a point $y\in E_{k-1}$
such that the distance from $y$ to $x$ is sufficiently small.
Fix $x\in E_k$. By Lemma \ref{small ball inside}, $B(x,r)\cap B$
contains a ball $\tilde B$ of radius $r/2$ if $0<r\le 2R$, and hence, by the
doubling property of $\mu$,
\begin{equation}\label{small ball}
\mu(B(x,r)\cap B)\ge\mu(\tilde B)\ge\mu(B)c_d^{-2}\Big(\frac{r}{2R}\Big)^s, 
\end{equation}
where $s=\log_2c_d$. Since $s$ can always be replaced by a larger number, we can assume that $s>1$. 
We also have the weak type estimate
\begin{equation}\label{weak type}
\mu(B\setminus E_{k-1})
=\mu\big(\{x\in B:\M_\lambda(x)>2^{k-1}\}\big)
< \frac C{2^{k-1}}\lambda(\tau B).
\end{equation}
Thus, in order to obtain the inequality 
$\mu(B(x,r) \cap B) > \mu(B \setminus E_{k-1})$, 
it is sufficient to require that
\[
\mu(B)c_d^{-2}\Big(\frac{r}{2R}\Big)^s \ge \frac{C}{2^{k-1}}\lambda(\tau B),
\]
that is,
\[
r \ge 2R \Big( \frac{C \lambda(\tau B)}{2^{k-1}\mu(B)}  \Big)^{1/s}.
\]
Let us thus define for each $k \in \z$
\[
r_k := 2R \Big( \frac{C \lambda(\tau B)}{2^{k-1}\mu(B)}  \Big)^{1/s}.
\]
Now, since $\mu(B(x,r_{k}) \cap B) > \mu(B \setminus E_{k-1})$, there is a $y \in B(x,r_k) \cap E_{k-1}$. The definition of $r_k$ and \eqref{u and ak point} then imply that
\[
a_k \le a_{k-1}+C_02^{k+1}r_k
=a_{k-1}+C_02^{k+2}R \Big( \frac{C \lambda(\tau B)}{2^{k-1}\mu(B)}  \Big)^{1/s}.
\]
Iterating the above estimate starting from some $k_0\in\z$, we get
\begin{equation}\label{ak}
\begin{aligned}
a_k
&\le a_{k_0}+ \sum_{i=k_0+1}^{k}CR2^{i(1-1/s)}
\Big( \frac{\lambda(\tau B)}{\mu (B)}  \Big)^{1/s} \\
&\le a_{k_0}+CR\Big( \frac{\lambda(\tau B)}{\mu (B)} \Big)^{1/s} \sum_{i=-\infty}^{k}2^{i(1-1/s)} \\
&\le a_{k_0}+CR\Big( \frac{\lambda(\tau B)}{\mu (B)}  \Big)^{1/s} 2^{k(1-1/s)}
\end{aligned}
\end{equation}
for each $k>k_0$. Note that $a_k \le a_{k_0}$ for each $k \le k_0$. 
Now, as mentioned before equation \eqref{small ball}, 
we must require that $r_{k_0+1} \le 2R$, that is,
\[
2R \Big( \frac{C\lambda(\tau B)}{2^{k_0}\mu(B)}  \Big)^{1/s} \le 2R,
\]
or equivalently
\[
\frac{C\lambda(\tau B)}{2^{k_0}} \le \mu(B).
\]
Let $k_0 \in \z$ be the smallest integer for which the above inequality holds. We then have
\[
2^{k_0} \ge \frac{C\lambda(\tau B)}{\mu(B)}
\]
and
\[
2^{k_0-1} < \frac{C\lambda(\tau B)}{\mu(B)}.
\]
By the doubling property of $\mu$ we thus have for some constant $C$
\begin{equation}\label{claim k0}
\frac{1}{C}\frac{\lambda(\tau B)}{\mu(\tau B)} \le 2^{k_0} \le C\frac{\lambda(\tau B)}{\mu (\tau B)}.
\end{equation}
We select the set $E$ discussed in the beginning of
the proof to be $E_{k_0}$. 
%Combining the weak estimate \eqref{weak type} and %\eqref{requirement k0}, and increasing the value of $k_0$ as %necessary, we see that we can assume that $\mu(E_{k_0})>0$, as %required.  
Note that $\mu(E_{k_0})>0$ because 
$\mu(B(x,r_{k_0+1})\cap B) > \mu(B \setminus E_{k_0})$ for any $x \in B$.
As mentioned in the beginning of the proof, we can now assume that 
$\operatorname{ess}\inf_{E_{k_0}}|u|=0$. 
Then, by the $C_02^{k_0+1}$-Lipschitz continuity of
$u$ in $E_{k_0}$, and \eqref{claim k0} we have the following estimate
for $a_{k_0}$:
\begin{equation}\label{ak0}
a_{k_0}=\sup_{E_{k_0}}|u|\le C_02^{k_0+1}\cdot2R\le 
CR\frac{\lambda(\tau B)}{\mu(\tau B)}.
\end{equation}
By writing $A_k=E_k\setminus E_{k-1}$ and using \eqref{u and ak} and \eqref{ak}, we  have  
%\begin{equation}\label{final}
\begin{align}
\frac12\int_B|u-u_B|\,d\mu
&\le\sum_{k=-\infty}^{\infty}a_k\mu(A_k)\nonumber\\
&\le \sum_{k=-\infty}^{k_0}a_{k_0}\mu(A_k)
+ \sum_{k=k_0+1}^{\infty}\Big(a_{k_0}+
 CR\Big(\frac{\lambda(\tau B)}{\mu(B)}\Big)^{1/s}2^{k(1-1/s)}\Big)\mu(A_k)\nonumber\\
&\le \sum_{k=-\infty}^{\infty}a_{k_0}\mu(A_k)
+ CR\Big(\frac{\lambda(\tau B)}{\mu(B)}\Big)^{1/s}\sum_{k=k_0+1}^{\infty} 2^{k(1-1/s)}\mu(B \setminus E_{k-1})\label{final},
\end{align}
%\end{equation}
where, by \eqref{ak0} and the doubling property of $\mu$,
\[
\sum_{k=-\infty}^{\infty}a_{k_0}\mu(A_k)
\le CR\frac{\lambda(\tau B)}{\mu(\tau B)}\mu(B)
\le CR\lambda(\tau B).
\]
Moreover, we estimate the last sum of \eqref{final} by using the weak type estimate \eqref{weak type} and \eqref{claim k0}:
\begin{align*}
\sum_{k=k_0+1}^{\infty} 2^{k(1-1/s)}\mu(B \setminus E_{k-1})
&\le \sum_{k=k_0+1}^{\infty}2^{k(1-1/s)}\frac{C\lambda(\tau B)}{2^{k-1}} \\
&\le C\lambda(\tau B)\sum_{k=k_0+1}^{\infty}2^{-k/s} \\
&\le C\lambda(\tau B)2^{-k_0/s} \\
&\le C\lambda(\tau B)\Big( \frac{\mu(\tau B)}{\lambda(\tau B)}  \Big)^{1/s}. \\
\end{align*}
Finally, using the doubling property of $\mu$, we get
\[
\frac12\int_B|u-u_B|\,d\mu \le CR\lambda(\tau B).
\]
Hence the claim follows --- we only have to note that the constants in the final form of the Poincar\'e type inequality will possibly be altered by the swaps between the metrics discussed in the beginning of the proof.
\end{proof}

\noindent \textbf{Acknowledgements.} The authors wish to thank Juha Kinnunen for helpful discussions and suggestions.


\begin{thebibliography}{10}
\bibitem{A1}
Ambrosio, L.: 
{\it Some fine properties of sets of finite perimeter in Ahlfors regular metric measure spaces}, 
Adv. Math. 159 (2001), no. 1, 51--67.
\bibitem{A}
Ambrosio, L.: 
{\it Fine properties of sets of finite perimeter in doubling metric measure spaces},
Set valued analysis 10 (2002), no. 2--3, 111--128.
\bibitem{AFP}
Ambrosio, L., Fusco, N., and Pallara, D.:
{\it Functions of bounded variation and free discontinuity problems}.
Oxford Mathematical Monographs. The Clarendon Press, Oxford University Press, New York, 2000.
\bibitem{AMP}
Ambrosio, L., Miranda, M. Jr., and Pallara, D.:
{\it Special functions of bounded variation in doubling metric measure spaces},
Calculus of variations: topics from the mathematical heritage of 
E. De Giorgi (2004), 1--45.
\bibitem{Cam}Camfield, C. S.:
{\it Comparison of BV norms in weighted Euclidean spaces and metric measure spaces},
Thesis (Ph.D.), 
University of Cincinnati, 2008. 
\bibitem{EG}Evans, L.C. and Gariepy, R.F.:
{\it Measure Theory and Fine Properties of Functions}.
CRC Press, 1992, Boca Raton-New York-London-Tokyo.
\bibitem{F} Federer, H.:
{\it Geometric Measure Theory}.
Grundlehren 153, Springer-Verlag, Berlin, 1969.
\bibitem{Gi} Giusti, E.: 
{\it Minimal surfaces and functions of bounded variation}.
Monographs in Mathematics, 80, Birkh\"auser Verlag, Basel, 1984.
\bibitem{Hj1} Haj\l asz, P.:
{\it Sobolev spaces on an arbitrary metric space},
Potential Anal. 5 (1996), no. 4, 403--415.
\bibitem{Hj2} Haj\l asz, P.:
{\it Sobolev spaces on metric-measure spaces},
In: ``Heat kernels and analysis on manifolds, graphs, and metric spaces'',
(Paris, 2002), 173--218, Contemp. Math. {\bf 338}, Amer. Math. Soc.
Providence, RI, 2003.
\bibitem{Hj3} Haj\l asz, P.: 
{\it A new characterization of the Sobolev space},
Studia Math. 159 (2003), no. 2, 263--275. 
\bibitem{HjK} Haj\l asz, P. and Koskela, P.:
{\it Sobolev met Poincar\'e},
Mem. Amer. Math. Soc. 145 (2000), no. 688.
\bibitem{HakKi} Hakkarainen, H. and Kinnunen, J.:
{\it The BV-capacity in metric spaces}, 
Manuscripta Math. 132 (2010), no. 1-2, 51--73.
\bibitem{He} Heinonen, J.:
{\it Lectures on analysis on metric spaces}, 
Universitext, Springer-Verlag, New York, 2001.
\bibitem{HeKo} Heinonen, J. and Koskela, P.:
{\it Quasiconformal maps on metric spaces with controlled geometry}, 
Acta Math. 181 (1998), 1--61.
%\bibitem{K}Keith, S.:
%{\it Modulus and the Poincar\'e inequality on metric measure spaces},
%Math. Z. 245 (2003), no. 2, 255--292. 
%\bibitem{KZ}Keith, S. and Zhong, X.:
%{\it The Poincar\'e inequality is an open ended condition},
%Ann. of Math. (2) 167 (2008), no. 2, 575--599.
\bibitem{KKST1}Kinnunen, J., Korte, R., Shanmugalingam, N., and Tuominen, H.: 
{\it Lebesgue points and capacities via the boxing inequality in metric spaces}, 
Indiana Univ. Math. J. 57 (2008), no. 1, 401--430.
\bibitem{KKST3}Kinnunen, J., Korte, R., Shanmugalingam, N., and Tuominen, H.: 
{\it A characterization of Newtonian functions with zero boundary values}, 
Calc. Var. Partial Differential Equations 43 (2012), no. 3-4, 507--528.
\bibitem{KKST2}Kinnunen, J., Korte, R., Shanmugalingam, N., and Tuominen, H.: 
{\it Pointwise properties of functions of bounded variation in metric spaces}, preprint 2012.
%\bibitem{Ko}Koskela, P. and MacManus, P.:
%{\it Quasiconformal mappings and Sobolev spaces},
%Studia Math. 131 (1998), no. 1, 1--17.
\bibitem{Mi}Miranda, M. Jr.:
{\it Functions of bounded variation on "good" metric spaces},
J. Math. Pures Appl. (9) 82 (2003), no. 8, 975--1004. 
\bibitem{S}Shanmugalingam, N.:
{\it Newtonian spaces: an extension of Sobolev spaces to metric measure spaces},
Rev. Mat. Iberoamericana 16 (2000), no. 2, 243--279.
\bibitem{Z}Ziemer, W.P.: 
{\it Weakly differentiable functions}.
Graduate Texts in Mathematics 120, Springer-Verlag, 1989.
\end{thebibliography}
\end{document}